\documentclass[11pt]{amsart}

\usepackage{amsmath}
\usepackage{latexsym}
\usepackage{amsfonts}
\usepackage{amssymb}
\usepackage{amscd}
\usepackage{enumerate}
\usepackage{mathrsfs}
\usepackage{color}

\newtheorem{theo}{Theorem}[section]

\newtheorem{proposition}[theo]{Proposition}
\newtheorem{corollary}[theo]{Corollary}
\newtheorem{conj}[theo]{Conjecture}

  \newtheorem{example}[theo]{Example}

\numberwithin{equation}{section}

\newcommand{\abs}[1]{\left| #1 \right|}

\newcommand{\N}{\mathbb{N}} 
\newcommand{\Z}{\mathbb{Z}} 
\newcommand{\R}{\mathbb{R}} 

\DeclareMathOperator{\supp}{supp}

\newcommand{\eps}{\varepsilon}
\newcommand{\ro}{\varrho}
\renewcommand{\phi}{\varphi}

\newcommand{\EE}{\mathscr E}
\newcommand{\DD}{\mathscr D}

\newcommand{\II}{\mathcal I}

\newcommand{\alle}[1]{\forall\, #1 \;}
\newcommand{\gibt}[1]{\exists\, #1 \;}

\newcommand{\Oma}{$\mathrm{(}\Omega\mathrm{)}$}
\newcommand{\DN}{$\mathrm{(DN)}$}

\newcommand{\vertiii}[1]{{\left\vert\kern-0.25ex\left\vert\kern-0.25ex\left\vert #1
    \right\vert\kern-0.25ex\right\vert\kern-0.25ex\right\vert}}

\newcommand{\III}{\|\kern-0.25ex |}  

\title[Extension operators for smooth functions]{Extension operators for smooth functions on compact subsets of the reals}

\author{Leonhard Frerick, Enrique Jord\'a, and Jochen Wengenroth}

\address{Fachbereich IV -- Mathematik, Universit\"{a}t Trier, D-54286 Trier, Germany}
\email{frerick@uni-trier.de}

\address{Departamento de Matem\'atica Aplicada, E. Polit\'{e}cnica Superior
de Alcoy, Universidad Polit\'ecnica de Valencia, Plaza Ferr\'andiz
y Carbonell 2, E-03801 Alcoy (Alicante), Spain}
\email{ejorda@mat.upv.es}

\address{Fachbereich IV -- Mathematik, Universit\"{a}t Trier, D-54286 Trier, Germany}
\email{wengenroth@uni-trier.de}

\subjclass[2010]{47A57, 46E25, 46A63}
\keywords{extension operator, spaces of smooth functions}

\begin{document}

\begin{abstract}
  We introduce sufficient as well as necessary conditions for a compact set $K$ such that there is a continuous linear extension operator from
  the space of restrictions $C^\infty(K)=\{F|_K: F\in C^\infty(\R)\}$ to $C^\infty(\R)$. This allows us to deal with examples of the form $K=\{a_n:n\in\N\}\cup\{0\}$
  for $a_n\to 0$ previously considered by Fefferman and Ricci as well as Vogt.
\end{abstract}

\maketitle

\section{Introduction}
For a compact subset $K$ of $\R^d$ we endow the space of smooth restrictions
\[
C^\infty(K)=\{F|_K: F\in C^\infty(\R^d)\}
\]
with the quotient topology of the Fr\'echet space $C^\infty(\R^d)$, i.e., with the sequence of norms
\begin{align*}
  \III f\III_n & =\inf\{\|F\|_n: F|_K=f\} \text{ where }
\\ & \|F\|_n=\sup\{|\partial^\alpha F(x)|: x\in\R^d,|\alpha|\le n\}.
\end{align*}
This is a Fr\'echet space and the restriction operator $R: C^\infty(\R^d)\to C^\infty(K)$, $F\mapsto F|_K$ is surjective.
We are interested in the question whether there exists a continuous linear {\it extension operator} $E:C^\infty(K)\to C^\infty(\R^d)$ which means that
$R\circ E= Id_{C^\infty(K)}$.
If this is the case we say that that $K$ has the {\it smooth extension property}. Till know, very few cases are understood, a remarkable result of \cite{BiMi} says
that {\it semicoherent subanalytic sets} have the smooth extension property.

%
%
%
%

One of the many difficulties with this question is that, for small $K$, there are no derivatives for $f\in C^\infty(K)$ so
that many classical analytical tools are not directly accessible. In one dimension -- and this is the case we concentrate on -- one can use divided differences as a substitute,
they were used, e.g., by Merrien \cite{Me} to prove $C^\infty(K)=\bigcap_{n\in\N} C^n(K)$ for $K\subseteq \R$.

This equality is no longer true in higher dimensions (for subanalytic sets $K$ it is equivalent to semicoherence \cite{BMP,BiMi}, an elementary example can be found in
\cite{Pa}) so that the recent deep result of Fefferman \cite{Fe} that, for every $n\in\N$, there is an extension operator $E_n:C^n(K)\to C^n(\R^d)$ is not directly applicable.

Instead of $C^\infty(K)$ one can consider the space of Whitney jets
\[
  \EE(K)=\{(\partial^\alpha F|_K)_{\alpha\in\N_0^d}: F\in C^\infty(\R^d)\}
\]
also endowed with the quotient topology from $C^\infty(\R^d)$. The corresponding question whether there is a continuous linear extension operator $\EE(K)\to C^\infty(\R^d)$
(then $K$ has the {\it Whitney extension property})
is not completely solved but much better understood than the smooth extension property, we refer to \cite{Fr} for many sufficient and necessary conditions.
It is proved in \cite[Remark 3.13]{Fr} that the existence of an extension operator for $\EE(K)$ implies $\EE(K)=C^\infty(K)$ (more precisely, $(\partial^\alpha F|_K)_{\alpha\in\N_0^d}
\mapsto F|_K$ gives an isomorphism), thus, the Whitney extension property implies the smooth extension property. Therefore, if $K$ is the closure of its interior and has Lipschitz boundary
\cite{St} or, more generally, not too sharp cusps \cite{BoMi,PaPl2} then it has both extension properties. The same holds for such porous sets as the Sierpi\'nski triangle \cite{FJW}.
An example with $C^\infty(K)=\EE(K)$ and without extension property is the sharp cusp $\{(x,y)\in\R^2: 0\le x\le 1, 0\le y\le \exp(-1/x)\}$ in \cite{Ti}.

However, if $C^\infty(K)$ is different from $\EE(K)$ much less is known. The extreme case of a singleton $K$ has the smooth extension property
(trivially, since $C^\infty(K)$ is one-dimensional) but not the Whitney extension property \cite{Mi}. The same holds for semicoherent subanalytic sets with empty interior.

For general sets without further analytical properties a characterization of the smooth
extension property seems to be far out of reach. In this article we continue the investigation of rather special sets $K=\{a_n:n\in\N\} \cup \{0\}$ for real sequences $a_n\to 0$
as considered by Fefferman and Ricci \cite{FeRi} and Vogt \cite{Vo}. In \cite{FeRi} it is shown
that for $a_n=n^{\alpha}$ with $\alpha < 0$ the set has the smooth extension property.

This was generalized by Vogt to decreasing sequences $a_n\to 0$ such that
\begin{enumerate}[(a)]
\item $a_n-a_{n+1}$ is decreasing,
\item $a_n/a_{n+1}$ is bounded, and
\item $a_n^q/(a_n-a_{n+1})$ is bounded for some $q\in \N$.
\end{enumerate}
In particular, $\{e^{-n}: n\in\N\}\cup \{0\}$ has the smooth extension property. However, examples like $a_n=1/\log(n)$, $a_n=e^{-n^2}$, or $a_n=e^{-2^n}$
are not covered by Vogt's approach. We are going to introduce several sufficient conditions as well as necessary ones in order to deal with such sequences.

Whereas Fefferman and Ricci gave an explicit construction of an extension operator Vogt as well as Bierstone and Milman
used the splitting theory for short exact sequences of Fr\'echet spaces and we will follow this strategy. For the ideal
$\II_K=\{F\in C^\infty(\R^d): F|_K=0\}$ we have a short exact sequence
\[
  0\to \II_K \to C^\infty(\R^d) \stackrel{R}{\to} C^\infty(K) \to 0
\]
and, by definition, $K$ has the smooth extension property if and only if the sequence splits in the category of
Fr\'echet spaces (the right inverses of $R$ are precisely the extension operators). The celebrated splitting theorem of Vogt and Wagner
\cite[chapter 30]{MV} says that it is sufficient to prove that $\II_K$ satisfies the topological invariant ($\Omega$) and $C^\infty(K)$ satisfies (DN) (has a dominating norm,
we will recall the definitions later on).
If $K$ has the smooth extension property we
can replace $E(f)$ by $\phi E(f)$ where $\phi$ is a smooth function with compact
support and equal to $1$ near $K$ to obtain an extension operator with values in $\DD(B)$ for some ball $B$. Since $\DD(B)$ satisfies (DN) and ($\Omega$) we conclude that
$K$ has the smooth extension property if and only if $\II_K \in (\Omega)$ and $C^\infty(K) \in$ (DN).

In section 2 we will show ($\Omega$) not only for $\II_K$ with compact subsets of $\R$ but for every closed ideal $\II$ in $C^\infty(\R)$  (the case
$\II=\{F\in C^\infty(\R): F^{(k)}|_K=0 \text{ for all } k\in\N_0\}$ is known and corresponds to $\EE(K)$).
Therefore, $K\subseteq \R$ has the smooth extension property if and only if $C^\infty(K)$ satisfies (DN), and we will prove a sufficient condition in section 3 and two necessary ones in section 4.
This allows us to show that $K=\{a_n:n\in\N\} \cup \{0\}$ has the smooth extension property for the very fast decaying sequence $a_n=e^{-n^2}$ but it does not for
extremely fast sequences like $a_n=e^{-2^n}$.

\section{Closed ideals in $C^\infty(\R)$}

In this section we will show that every closed ideal of $C^\infty(\R)$ satisfies property ($\Omega$) of Vogt and Wagner \cite{VoWa,MV}.
This is possible because a simple instance of Whitney's spectral theorem \cite{Wh2,Ma} gives a full description of all closed ideals $\II$: There is a {\it multiplicity function}
$\mu:\R\to\N_0 \cup \{\infty\}$ such that
\[
  \II=\{f\in C^\infty: f^{(j)}(x)=0 \text{ for all $x\in\R$ and $0\le j <\mu(x)$}\}
\]
(for $\mu(x)=0$ there is thus no condition on $f(x)$). We will prove ($\Omega$) in the following form (which is equivalent to the submultiplicative inequalities for the dual norms
in the definition in \cite[chapter 29]{MV}):
A Fr\'echet space $X$ with fundamental sequence of seminorms $\|\cdot\|_n$ satisfies ($\Omega$) if
\[
\alle{n\in\N} \gibt{m\ge n} \alle{k\ge m} \gibt{s\in\N, c>0} \alle{\eps>0}
\text{ every $x\in X$ with } \|x\|_m\le 1
\]
\[
\text{can be written as } x=x-y+y  \text{ such that}
\|x-y\|_n\le \eps \text{ and } \|y\|_k\le c\eps^{-s}.
\]
Note that these are approximation problems with respect to the $n$-th norm requiring specific bounds for the $k$-th norm of the approximants.
We are going to solve these problems in $\II$ for the seminorms $\| f\|_n=\sup\{|f^{(j)}(x)|: |x|\le n, 0\le j\le n\}$ by using rather classical
approximation properties of Hermite interpolation polynomials.
Below, we will explain that the following theorem generalizes in a certain sense Merrien's result mentioned above.

\begin{theo}\label{Omega}
Every closed ideal $\II$ of $C^\infty(\R)$ satisfies \Oma.
\end{theo}

\begin{proof}
We take a multiplicity function $\mu$ for $\II$ as above.
For $n\in\N$ we will prove the \Oma-condition with $m=2n+1$ and $s=k$ (for given $k\ge m$). Even the constants $c=c_k$ will turn out to be independent of the ideal.
In the following, $c$ and $c_k$ always denote constants which are independent of $f$ and $\eps>0$ and may vary at different occurrences.

We start with a partition of unity of the form
  \[
    \sum_{\ell\in\Z} \phi(x-\ell)=1
  \]
where $\phi$ is a positive smooth function with support in the interval $(-1,1)$ so that, for each $x\in \R$, at most two terms of the series do not vanish. Such a partition
can be seen, e.g., in \cite[Theorem 1.4.6]{Ho}. For $c_k=\|\phi\|_k$ and $\eps>0$ the scaled functions $\ro_{\eps,\ell}(x)=\phi(x/\eps-\ell)$ then satisfy
$\|\ro_{\eps,\ell}\|_k\le c_k/\eps^{k}$, $\supp \ro_{\eps,\ell} \subseteq I_\ell=((\ell-1)\eps,(\ell+1)\eps)$, and $\sum_\ell \ro_{\eps,\ell}(x)=1$ with again
at most two non-vanishing terms.
Given now  $f\in \II$ with $\|f\|_m\le 1$ and $\eps\in (0,1)$ (for $\eps\ge 1$ there is the trivial \Oma-decomposition $f=f-0+0$) we make the following ansatz:
\[
  g=\sum_{\ell\in\Z} \ro_{\eps,\ell} g_\ell
\]
with polynomials $g_\ell$ of degree $m$ to be chosen in a way such that $\ro_{\eps,\ell} g_\ell \in\II$, $\|f-g\|_n \le \eps$, and $\|g\|_k\le c_k \eps^{-k}$.

For the choice of $g_\ell$ we distinguish two cases depending on the number $N_\ell$ of prescribed zeroes (counted with multiplicities)
of the ideal in $I_\ell$, that is, $N_\ell=\sum_{x\in I_\ell} \mu(x)$.
If $N_\ell > m+1$ or $|\ell| > (n+1)/\eps$ we just put $g_\ell=0$. Otherwise, we increase one of the $\mu(x)$ for an arbitrarily chosen $x\in I_\ell$ so that
$N_\ell=m+1$, and take $g_\ell$ as the unique solution of the Hermite interpolation problem with data $\{(x,f^{(j)}(x)): x\in I_\ell, 0\le j< \mu(x)\}$. This means that
$g_\ell$ is a polynomial of degree $m$ such that $g_\ell^{(j)}(x)=f^{(j)}(x)$ for all $x\in I_\ell$ and $0\le j <\mu(x)$.
Since $f\in\II$ the polynomial $g_\ell$ satisfies in $I_\ell$ all necessary conditions for belonging to $\II$. Therefore, $\ro_{\eps,\ell} g_\ell \in \II$ and hence $g\in \II$.

We will first estimate $\|g\|_k$ with $k\ge m$ for which it is enough to estimate $\|\ro_{\eps,\ell}g_\ell\|_k$ for each $\ell$ with $g_\ell\neq 0$.
For $j\le k$ we apply Leibniz' rule and
the inequalities for the derivatives of $\ro_{\eps,\ell}$ from above to get, for all $x\in I_\ell$,
\begin{align*}
  \left|\left(\ro_{\eps,\ell} g_\ell\right)^{(j)}(x)\right| & \le c_k \eps^{-k} \sup\{|g_\ell^{(i)}(x)| :0 \le i \le k\} \\
  & =  c_k \eps^{-k} \sup\{|g_\ell^{(i)}(x)| :0 \le i \le m\}
\end{align*}
because $g_\ell$ is a polynomial of degree $m$. To estimate the derivatives of $g_\ell$ we need the concrete form of the Hermite interpolation
polynomials and, in order to be consistent with
the commonly used notation as, e.g., in \cite[chapter 4,\S 6]{DevoLo}, we fix an ordered vector $(x_0,\ldots,x_m)$ in which each $x\in I_\ell$ appears $\mu(x)$ times. Then
\[
  g_\ell(x)=\sum_{s=0}^m f[x_0,\ldots,x_s](x-x_0)\cdots(x-x_{s-1})
\]
with the (generalized) divided differences as coefficients. For real valued $f$ (which we may assume, of course) there are $\xi_s\in I_\ell$ such that
$f[x_0,\ldots,x_s]=f^{(s)}(\xi_s)/s!$. Since $|x-x_j|\le 2\eps \le 2$ for $x\in I_\ell$ we thus get $|g_\ell^{(j)}|\le c\|f\|_m$ on $I_\ell$ and hence
\[
  \|g\|_k \le c_k\eps^{-k}\|f\|_m.
\]
It remains to show $\|f-g\|_n\le c\eps$ with a constant independent of $\eps$ (which afterward can be removed by applying the obtained decomposition for $\tilde \eps=\eps/c$), and
because of $f-g= \sum \ro_{\eps,\ell}(f-g_\ell)$ it is again enough to estimate each term. We do this for the case where $g_\ell\neq 0$, the other one is similar
(and even a particular case of the following arguments by choosing $x_0,\ldots,x_m$ arbitrarily among the zeroes of $\II$ in $I_\ell$).

Given $x_0,\ldots,x_m$ as above we write $H$ for the linear map assigning to $h\in C^m(I_\ell)$ its Hermite interpolation polynomial for the data
$\{(x,h^{(j)}(x)): x\in I_\ell, 0\le j< \mu(x)\}$. Then $H$ is a projector onto the subspace of polynomials up to degree $m$. For the midpoint $y=\ell\eps$ of $I_\ell$
and the Taylor polynomial $T_y^mf$ of degree $m$ around $y$ we thus have on $I_\ell$
\[
  f-g_\ell=f-H(f)=(f-T_y^mf) + H(T_y^mf - f).
\]
By Taylor's theorem, the derivatives up to order $n$ of the first term are less than $2\|f\|_m\eps^{m-n}$. The $j$-th derivative of the second term is
\begin{align*}
  \left( H(T_y^mf -f)\right)^{(j)}(x) &= \sum_{s=0}^m (T_y^mf -f)[x_0,\ldots,x_s] \left( (x-x_0)\cdots(x-x_{s-1})\right)^{(j)} \\
  &= \sum_{s=j}^m (T_y^mf -f)^{(s)}(\xi_s) \left( (x-x_0)\cdots(x-x_{s-1})\right)^{(j)}
\end{align*}
with $\xi_s\in I_\ell$. On $I_\ell$ we thus get again by Taylor's theorem and Leibniz' formula a constant $c$ (depending only on $m$) with
\begin{align*}
  \left|\left( H(T_y^mf -f)\right)^{(j)}(x)\right| & \le c \sum_{s=j}^m  \left| (T_y^{m-s}f^{(s)} -f^{(s)})(\xi_s)\right| (2\eps)^{s-j} \\
  & \le c\sum_{s=j}^m \|f\|_m \eps^{m-s}(2\eps)^{s-j} \le \tilde c\eps^{m-j} \|f\|_m \le \tilde c \eps^{m-n}.
\end{align*}
Combining this with $|\ro_{\eps,\ell}^{(j)}(x)|\le c_n\eps^{-n}$ and $m=2n+1$ we finally get
\[
  \|f-g\|_n \le c \eps^{m-2n}= c\eps. \qedhere
\]
\end{proof}

In the proof above we did not use that $f$ is smooth but only that $f\in C^{2n+1}(\R)$. In particular, we have shown for any  set $K\subseteq \R$ and
\[
  \II^n_K=\{f\in C^n(\R): f|_K=0\}
\]
that every $f\in \II_K^{2n+1}$ can be decomposed as $f=f-g+g$ with $g\in \II^\infty_K$ and $\|f-g\|_n <\eps$.
(The proof even simplifies a bit because one does not need the estimate for $\|g\|_k$ and, because all zeroes of the ideals are simple, one can
use Lagrange instead of Hermite interpolation.) Expressed differently, the closure of $\II^\infty_K$ in $\II_K^n$ contains $\II_K^{2n+1}$.
This {\it reducedness} of the projective spectrum $(\II_K^n)_{n\in\N_0}$ allows us to apply the abstract Mittag-Leffler procedure, see, e.g., \cite[section 3.2]{We}:
For each $n\in\N_0$ we have a short exact sequence
\[
  0\to \II_K^n \to C^n(\R) \to C^n(K) \to 0
\]
and the projective limit (with respect to the inclusions as spectral maps) of these sequences is
\[
  0\to \II^\infty_K \to C^\infty(\R) \to \bigcap_{n\in\N_0} C^n(K) \to 0.
\]
In general, the projective limit of exact sequences need not be exact at the last spot and the non-exactness is measured by the first derivative of the
projective limit functor. The abstract Mittag-Leffler theorem \cite[Theorem 3.2.1]{We} now states that this first derivative vanishes for reduced spectra.
As this is case here we get that the limit is indeed exact. We have thus a new proof of the following result from \cite{Me}:

\begin{theo}[Merrien]\label{Merrien}
  $C^\infty(K)=\bigcap\limits_{n\in\N} C^n(K)$ for every set $K\subseteq \R$.
\end{theo}

We do not know any compact set $K\subseteq \R^d$ such that the vanishing ideal  $\II_K=\{f\in C^\infty(\R^d): f|_K=0\}$ does not satisfy \Oma.
To be concrete, we thus state a very optimistic conjecture:

\begin{conj}
  Every vanishing ideal $\II_K$ in $C^\infty(\R^d)$ satisfies \Oma.
\end{conj}





\section{A Sufficient condition for (DN)}

In this section we prove a sufficient condition for a compact set $K\subseteq \R$ such that $C^\infty(K)$ has a dominating norm.

Let us recall that for a  Fr\'echet space $X$ with fundamental sequence of seminorms $\|\cdot\|_n$ the $n$-th seminorm is dominating if
\[
  \alle{m\ge n} \gibt{k\ge m, c>0} \alle{x\in X}
\]
\[
  \|x\|^2_m\le c\|x\|_k \|x\|_n.
\]
An equivalent condition is $\gibt{\vartheta \in (0,1)} \alle{m\ge n}\gibt{k\ge m,c>0}$ such that $\|x\|_m\le c\|x\|_k^\vartheta \|x\|_n^{1-\vartheta}$
(the passage from the given $\vartheta$ to $\vartheta=1/2$ is done by iterating the latter condition if $\vartheta >1/2$ and it is trivial for $\vartheta<1/2$)
which is satisfied if (and only if) we have
\[
\gibt{\sigma\ge 1}\alle{m\in\N}\gibt{k\ge m,c>0,\eps_k\in(0,1)} \alle{x\in X,0<\eps<\eps_k}
\]
\[
  \|x\|_m\le c(\eps \|x\|_k + \eps^{-\sigma}\|x\|_n).
\]
Indeed, by increasing the constant we get the inequality for all $\eps\in(0,1)$ and minimizing the right hand side then implies the submultiplicative inequality with
$\vartheta=\sigma/(1+\sigma)$ (and a different constant).

\medskip

Just for convenience, we will slightly modify (a finite number of) the seminorms of $C^\infty(\R)$ from the previous section: Given a compact set $K\subseteq \R$ we
set $\|F\|_n=\sup\{|F^{(j)}(x)|: 0\le j\le n, x\in K\cup [-n,n]\}$.

For a closed ideal with multiplicity function $\mu$ whose zero set $Z(\II)=\{x\in\R:\mu(x)>0\}$ is contained in $K$ all quotient seminorms
\[
  \III f\III_n=\inf\{\|F\|_n: F\text{ represents } f\}
\]
(where, of course, $F$ represents $f$ if $f$ is the equivalence class $F+\II$ of $F$) are in fact norms on $C^\infty(\R)/\II$.

As mentioned in the introduction, theorem \ref{Omega} implies that a closed ideal $\II$ is complemented in $C^\infty(\R)$ if  $C^\infty(\R)/\II$
satisfies \DN, i.e., it has a dominating norm. Using the specific form of closed ideals and a partition of unity one easily sees that the assumption
that $Z(\II)$ is compact is no restriction of generality.

As before, for a closed ideal $\II$ with multiplicity function $\mu$ and $I\subseteq \R$ we call $\sum\limits_{x\in I} \mu(x)$ the number of zeroes of $\II$ in $I$.

It is quite natural to expect that (DN) for the quotient $C^\infty(\R)/\II$ depends on the way the points of
$Z(\II)$ accumulate. The theorem below describes a kind of thickness of $K$ near its points which are not ``very isolated''
expressed in terms of a local {\it Markov equality}.
Several versions of it appeared in the context of Whitney extension operators, e.g.,
in  \cite{PaPl2,BoMi,Fr,FJW1,FJW}. A more geometric condition will be derived afterwards.

\begin{theo}
\label{sufficient}
Let $\II \subseteq C^{\infty}(\R)$ be a closed ideal such that $K=Z(\II)$ is compact. The norm
 $\III\cdot\III_n$ is dominating in $C^{\infty}(\R)/\II$ provided that the following condition holds:
\[
  \gibt{r\ge 1,\gamma\ge 1} \alle{m,k\in\N} \gibt{c>0,\eps_k\in (0,1)} \alle{x\in K,\, 0<\eps<\eps_k}
\]
either $(x-\eps^r,x+\eps^r)$ contains at most $n+1$ zeroes of $\II$ or there is $y\in K\cap (-\eps,\eps)$ such that
$$
|P^{(j)}(y)|\leq \frac{c}{\eps^{\gamma m}} \sup\{|P(t)|: t\in K\cap (y-\eps,y+\eps)\}
$$
for all polynomials $P$ of degree $\leq k$ and all $j\in\{0,\ldots,m\}$.
\end{theo}

\begin{proof}

We may assume that $r$ and $\gamma$ are integers. Given $m\ge n$ we set $\tilde m = (r+2)m$ and $k = (r+1)m +\gamma \tilde m$. The condition applied to $\tilde m$ and $k$
responds with a constant $c$ and $\eps_k>0$. We fix $f\in C^\infty(\R)/\II$ and $\eps\in (0,\eps_k)$.

As in the proof of theorem \ref{Omega} the constants below may vary from one occurrence to the other but are always independent of $\eps$, $x\in K$, and $f$.
We take a partition of unity $\ro_1,\ldots,\ro_M$ on $K$, such that
$\supp(\ro_\ell)\subseteq I_\ell = (x_\ell-\eps^r,x_\ell+ \eps^r)$, every $x\in\R$ belongs to at most two $I_\ell$, and
$$
|\ro^{(j)}_{\ell}(x)|\leq c \eps^{-rj}.
$$
We will construct a representative $F=\sum_{\ell=1}^{M}\ro_\ell g_\ell$ of $f$ with suitable $g_\ell$
by distinguishing two cases:
\begin{enumerate}[(i)]
\item If  the number  $N_\ell=\sum\limits_{x\in I_\ell} \mu(x)$ of zeroes in $I_\ell$ is $\le n+1$ we let $g_\ell$ be the  polynomial of degree $N_\ell-1$ interpolating
the  values and derivatives of $f$ up to order $\mu(x)-1$ for all in $x\in I_\ell$
(note that this does not depend on the representative $F_0$, i.e., $f^{(j)}(x) =F_0^{(j)}(x)$ is well-defined for $j < \mu(x)$).

\item Otherwise we choose $G\in C^{\infty}(\R)$ representing $f$ such that $\|G\|_k\leq 2\vertiii{f}_k$,
which is possible since $\vertiii{f}_k$ is the infimum of all such $\|G\|_k$  (of course we only have to deal with the case $\vertiii{f}_k\neq 0$).
We put $g_\ell=G$.
\end{enumerate}

Since all $g_\ell$ represent $f$ in $I_\ell$ and $\sum_{\ell=1}^{M}\ro_\ell=1$ on $K=Z(\II)$  we get that $F=\sum_{\ell=1}^{N}\ro_\ell  g_\ell$ represents $f$.

We will estimate the derivatives up to order $m$ of the terms $\ro_\ell g_\ell$ of $F$.

Case (i) is similar to the proof of theorem \ref{Omega}. We put the $N_\ell$ zeroes of $\II$ in $I_\ell$
in a vector $(x_0,\ldots,x_{N})$ with $N=N_\ell-1$ and write for $x\in I_\ell$
$$
g_\ell(x)=f[x_0]+f[x_0,x_1](x-x_0)+\cdots f[x_0,x_1,\ldots,x_N](x-x_0)\cdots (x-x_{N-1}).
$$
Given any representative $F_0$ of $f$ there are $\xi_s\in I_\ell$ with
$$
f[x_0,x_1,\ldots,x_s]=\frac{F_0^{(s)}(\xi_s)}{s!}
$$
so that $|g_\ell^{(j)}(x)| \le C\|F_0\|_n$ for $x\in I_\ell$ and all $j\le m$ (for $j>n$ the derivative is $0$ because $g_\ell$ is a polynomial of degree $N\le n$).
Combined with Leibniz' rule and the estimates for the derivatives of $\ro_\ell$ we get for $x\in \R$ and $j\le m$
\[
  |(\ro_\ell g_\ell)^{(j)}(x)|\le c \eps^{-rm} \|F_0\|_n.
\]
Passing to the infimum over all representations the last term can be replaced by $c\eps^{-rm} \III f\III_n$.

In case (ii) we choose  a point $y \in K \cap (x_\ell-\eps,x_\ell+\eps)$
where the Markov type inequality is satisfied for derivatives up to order $\tilde m$.

For $j\in \{0,\ldots,m\}$ and $x\in I_\ell$, Taylor's theorem gives
$$
g_{\ell}^{(j)}(x)=G^{(j)}(x)=\sum_{\beta\leq \tilde m-j-1}\frac{G^{(j+\beta)}(y)}{\beta!}(x-y)^{\beta}+\frac{G^{(\tilde m)}(\xi)}{(\tilde m-j)!}(x-y)^{\tilde m-j}
$$
for some $\xi$ between $x$ and $y$. From $|x-y|\leq 2\eps$  we then get, by the same estimate for the derivatives of $\ro_\ell$ as above,
\[
\left| (\ro_\ell g_{\ell})^{(j)}(x)\right|
    \leq c \eps^{-rm}\left(\sup_{0\leq \beta \leq \tilde m-1}|G^{(\beta)}(y)|+\eps^{\tilde m -m} \|G\|_{\tilde m}\right)
\]

Applying the local Markov inequality to the Taylor polynomial $T^k_yG$ of $G$ around $y$ gives for $\beta\le \tilde m \le k$
\begin{align*}
|G^{(\beta)}(y)| & =|(T_{y}^{k}G)^{(\beta)}(y)|\leq c \eps^{-\gamma \tilde m }\sup_{\omega\in K,|y-\omega|<\eps}|(T_y^k G)(\omega)|
\\
& \leq c \eps^{-\gamma \tilde m } \left(\sup_{|y-\omega|<\eps}|(T_y^k G)(\omega)-G(y)|+ \sup_{\omega\in K}|G(\omega)|\right)
\\
& \leq c \eps^{-\gamma \tilde m }  \left( \eps^k \|G\|_k  + \III f\III_0\right)
\end{align*}
because of Taylor's theorem and the fact that all representatives of $f$ coincide on $K$.
Combining both inequalities and using $\tilde m-(r+1)m = m$, $k-\gamma \tilde m -rm = m$, and $\|G\|_{\tilde m} \le \|G\|_k\le 2\III f\III_k$  we get
\begin{align*}
\left| (\ro_\ell g_{\ell})^{(j)}(x)\right| &\le c\left( \eps^{k-\gamma\tilde m -rm}\|G\|_k + \eps^{-\gamma \tilde m-rm}\III f\III_0 +\eps^{\tilde m-m-rm}\|G\|_{\tilde m}\right) \\
& \le c\left( \eps^m \III f\III_k + \eps^{-\gamma \tilde m-rm}\III f\III_0\right).
\end{align*}
By the definition of $\tilde m= (r+2)m$ we get, with $\sigma=\gamma(r+2)+r$, in both cases the inequality
\[
 \left| (\ro_\ell g_{\ell})^{(j)}(x)\right| \le c \left( \eps^m \III f\III_k + \eps^{-\sigma m}\III f \III_0\right).
\]
Summing over $\ell$ and taking the supremum of all $x\in K$ we have thus proved
\[
  \III f\III_m \le \|F\|_m \le c\left( \eps^m \III f\III_k + \eps^{-\sigma m}\III f\III_n\right)
\]
for all $\eps\in(0,\eps_k)$. Replacing $\eps$ by $\eps^{1/m}$ we obtain
\[
  \III f\III_m \le c\left( \eps \III f\III_k + \eps^{-\sigma }\III f\III_n\right)
\]
for all $\eps\in(0,\eps_k^m)$ which proves that $\III \cdot\III_n$ is a dominating norm.
 \end{proof}

It is clear that the Markov type inequality cannot hold for polynomials of degree $k$ if $K\cap (y-\eps,y+\eps)$ has strictly less that $k$ points.
On the other hand, we will show that it is sufficient to find $k$ points in the intersection which are {\it regularly distributed} so that the
minimal distance between two points is comparable to the maximal distance. We thus get a sufficient geometric condition for the smooth extension property
which can be evaluated in concrete cases.

 \begin{theo}
\label{lagrange}
Let $K$ be a compact subset of $\R$ and $n\in \N$ such that
\[
\gibt{r\ge 1,\gamma\ge 1} \alle{ m\in\N,k\in\N} \gibt{c>0,\eps_k>0} \alle{\eps\in (0,\eps_k), x\in K}
\]
whenever $(x-\eps^r,x+\eps^r)$ contains strictly more  than $n+1$ points of $K$ then
\[
\gibt{ y_0,\ldots, y_k\in K\cap (x-\eps,x+\eps)} \text{ with }  \frac{\sup_{0\leq i,\nu\leq k}|y_i-y_\nu|^{k-m}}{\inf_{i\neq \nu}|y_i-y_\nu|^k}\leq \frac{c}{\eps^{\gamma m}}.
\]
Then $K$ has the smooth extension property.
\end{theo}

\begin{proof}
With the quantifiers as above, $x\in K$ such that $K\cap (x-\eps^r,x+\eps^r)$ has more than $n+1$ points,
$y_0,\ldots,y_k$ as above, and a polynomial $P$ of degree $k$ we write it with Lagrange interpolation as
\[
P(t)=\sum_{\nu=0}^{k}P(y_\nu)L_\nu(t), \text{ where }
  L_\nu(t)=\frac{\prod_{i=0,i\neq \nu}^{k}(t-y_i)}{\prod_{i\neq \nu}(y_\nu-y_i)}.
\]
For $j\in\{0,\ldots,m\}$ we then have
\[
  P^{(j)}(t)=\sum_{\nu=0}^{k}P(y_\nu)\sum_{S\subset\{0,\ldots,k\}\setminus\{\nu\},|S|=k-j}\frac{\prod_{i\not\in S}(t-y_i)}{\prod_{i\neq \nu}(y_\nu-y_i)}.
\]
Denoting the quotient in the statement of the proposition by $q$ we thus get for $j\in\{0,\ldots,m\}$
\[
  \left| P^{(j)}(y_0)\right|\leq \sum_{\nu=0}^{n}|P(y_\nu)| c q\leq \frac{c}{\eps^{\gamma m}}\sup_{\omega\in K,\ |\omega-y_0|<\eps}|P(\omega)|.
\]
We have thus verified the required inequalities of theorem \ref{sufficient} for $y=y_0$.
\end{proof}

 It is interesting to note that the conditions of theorems \ref{sufficient} and \ref{lagrange} are both stable under unions, i.e.,
 if it is satisfied for $K$ and $L$ with $n_K$ and $n_L$ then it is also fulfilled for $K\cup L$ with $n=n_K +n_L+1$.
 We do not know however if the smooth extension property is stable under unions.

 \medskip

We will now apply theorem \ref{lagrange} to sets with only one accumulation point of the form $K=\{0\}\cup\{a_\ell:\ell \in\N_0\}$ for a decreasing null
sequence such that the sequence of differences
$d_\ell=a_\ell-a_{\ell+1}$ is decreasing. Since this monotonicity is equivalent to $a_\ell \le (a_{\ell-1}+a_{\ell+1})/2$ such sequences are called convex.
The following proposition improves Vogt's results \cite{Vo} as well as those of Fefferman and Ricci \cite{FeRi}.

\begin{proposition}\label{sequences}
Let $(a_\ell)_{\ell\in\N_0}$ be a decreasing convex null sequences such that, for every $p>1$, the sequence $a_\ell^p/a_{\ell+1}$ is bounded.\\
Then $K=\{0\}\cup\{a_\ell:\ell \in\N_0\}$ has the smooth extension property.
\end{proposition}

\begin{proof}
  We will verify the condition of theorem \ref{lagrange} for $n=0$, $r=\gamma=2$, and $\eps_k=1/4k$. Let us thus fix $m,k\in\N$, $\eps\in(0,\eps_k)$ and $x=a_\ell\in K$
  such that $(a_\ell-\eps^2,a_\ell+\eps^2)$ contains at least two elements of $K$ so that $d_\ell=a_\ell -a_{\ell+1}<\eps^2<\eps/4k$.

  For the construction of $y_0,\ldots,y_k$ we distinguish two cases depending on whether the limit point $0$ of the sequences belongs to $(a_\ell-\eps,a_\ell+\eps)$.
  If it does not not, i.e., $a_\ell>\eps$, we set $y_0=a_\ell$ and define $y_1,\ldots,y_k$ recursively: If $y_0,\ldots,y_{i-1}$ are already defined we let
  \[
    \ell(i)=\min\{j\in\N: y_{i-1}-a_j \ge \eps/4k\} \text{ and } y_i= a_{\ell(i)}.
  \]
We have to show that $y_1,\ldots,y_k$ are indeed well-defined elements of $K\cap (a_\ell-\eps,a_\ell+\eps)$. Since $d_j$ is decreasing we have $d_j \le d_\ell <\eps/4k$,
and for all $\ell(j)$ which are already defined this implies $y_{j-1}-y_j < \eps/{2k}$. Hence
\[
  y_0-y_{i-1}=(y_0-y_1)+\cdots + (y_{i-2}-y_{i-1}) < (i-1) \eps/2k < \eps/2
\]
which yields $y_{i-1}> y_0-\eps/2 > \eps/2$. Since $a_\ell$ tends to $0$ we get that $\ell(i)$ is well-defined.

By this construction, we have $|y_i-y_\nu|\ge\eps/4k$ for all distinct $i,\nu\in\{0,\ldots,k\}$ as well as $|y_0-y_k|\le \eps/2$.
This implies
\[
   \frac{\sup_{0\leq i,\nu\leq k}|y_i-y_\nu|^{k-m}}{\inf_{i\neq \nu}|y_i-y_\nu|^k}\leq \left(\eps/2\right)^{k-m} \left(4k/\eps\right)^k \le c_k \eps^{-m}.
\]
The condition of theorem \ref{lagrange} is thus satisfied even with $\gamma=1$.

In the second case $0\in (a_\ell-\eps,a_\ell+\eps)$, this neighbourhood contains all terms of the sequence with $a_j<\eps$.
We define $p>1$ by the equation $kp^{2k+1} = k+1$ and get a constant $c\ge 2$ (depending only on $k$) with $a_j^p/a_{j+1} \le c$ for all $j\in \N$.
We thus get $\ro\in (0,1/2)$ with $a_{j+1}\ge \ro a_j^p$. We claim that we then have
\[
  \alle{\delta\in (0,a_0)} \gibt{i\in\N} \text{ with } a_i\in \left[\ro\delta^p,\delta\right).
\]
Indeed, if $i\in\N$ is maximal with $a_{i-1}\ge \delta$ we have $a_i <\delta$ as well as $a_i\ge \ro a_{i-1}^p\ge\ro\delta^p$.

Defining $t_0=\eps$ and $t_{j+1}=\ro t_j^p$, i.e., $t_{j+1}=\ro^{p^j+p^{j-1}+\cdots+1} \eps^{p^{j+1}}$ we get a partion of $(0,\eps)$ into subintervals
$[t_{j+1},t_j)$ each of which containing elements of $K$ (of course, we may assume $\eps<a_0$). We choose $y_j\in K$ from every second interval, i.e.,
$y_j\in [t_{2j+1},t_{2j})$ for $j\in\{0,\ldots,k\}$. For $0\le i < j \le k$ we then have
\[
  y_i-y_j \ge y_{j-1} - y_j \ge t_{2j-1} - t_{2j} =t_{2j-1}-\ro t_{2j-1}^p.
\]
Since $t_{2j-1} < \eps<1$ and $p>1$, this is $\ge (1-\ro) t_{2j-1} \ge (1-\ro) t_{2k-1}$.
The explicit formula for $t_j$ thus gives a constant $\alpha>0$ (depending via $\ro$ and $p$ only on $k$) such that
\[
  y_i-y_j \ge \alpha \eps^{p^{2k-1}}.
\]

From this and the choice of $p$ we finally get, for $\gamma\geq 2$
\[
  \frac{\sup_{0\leq i,\nu\leq k}|y_i-y_\nu|^{k-m}}{\inf_{i\neq \nu}|y_i-y_\nu|^k}\leq c\eps^{k-m-kp^{2k+1}} = c\eps^{-(m+1)} \le c\eps^{-\gamma m}. \qedhere
\]
\end{proof}

Below we will show in example \ref{exampleexponential} that we cannot replace the quantifier {\it for all $p>1$} by a fixed $p>1$. 
Monotonicity and convexity of the following examples are easily checked by calculus. It is of course enough to have these properties for large $\ell$.

\begin{example}
The set $K=\{0\}\cup \{a_\ell:\ell\in\N\}$ has the smooth extension property in each of the following cases:
\begin{enumerate}[(1)]
  \item $a_\ell= \log(\ell+1)^\alpha/\ell^\beta$ for $\alpha\in\R$ and $\beta>0$.
  \item $a_\ell= \exp(-\ell^\alpha)$ for $\alpha>0$.
  \item $a_\ell= 1/\log(\ell+1)^\alpha$ for $\alpha>0$.
\end{enumerate}
\end{example}

Vogt's results mentioned in the introduction yield the cases (1) and (2) for $\alpha \le 1$ but they do not cover the other cases. The situations
in (2) and (3) exhibit extremely fast and slow decay, respectively, so that
one is tempted to believe that any set of the form $K=\{0\}\cup \{a_\ell:\ell\in\N\}$ with a decreasing null sequence might have the smooth extension property.
As we will show in the next sections this is not the case.



\section{Necessary conditions}
\subsection{A geometric necessary condition}
To obtain a geometric necessary condition we will need a result of Whitney \cite{Wh} describing $C^{\infty}(K)$ for a compact set $K\subseteq \R$
in terms of divided differences. For $f:K\to\R$ and distinct
$x_0,\ldots,x_n$, the divided differences are given by $f[x_0]=f(x_0)$ and
\[
  f[x_0,\ldots,x_n]=\frac{f[x_0,\ldots,x_{n-1}]-f[x_1,\ldots,x_n]}{x_0-x_n}.
\]
We define
\[
  |f|_n=\sup\{|f[x_0,\ldots,x_{j}]|:\ 0\leq j\leq n,\,  x_0,\ldots,x_j \in K \text{ distinct}\}.
\]
Whitney's theorem says that $f:K\to\R$ belongs to $C^{n}(K)$ if and only if the $n$-th divided difference map is uniformly continuous, i.e.,
for all $\eps>0$ there is $\delta>0$  such that for all
$x\in K$, $x_0,\ldots,x_n\in K\cap (x-\delta,x+\delta)$ distinct,
and $y_0,\ldots,y_n\in K\cap (x-\delta,x+\delta)$ distinct we have
\[
  |f[x_0,\ldots,x_n]-f[y_0,\ldots,y_n]|<\eps.
\]
From this it is easy to obtain that $|\cdot|_n$ is a complete norm on $C^n(K)$. Since by Merrien's theorem \ref{Merrien}
$$
C^{\infty}(K)=\bigcap_{n\in\N}C^n(K)
$$
this implies that
the system of norms $\{|\cdot|_n:\ n\in\N\}$ defines the Fr\'echet space topology of $C^{\infty}(K)$.

 \begin{theo}\label{necessarypoint}
 If $C^{\infty}(K)$ has a dominating norm then there exist $n\in\N$, $s\in\N$ such that for all $\eps\in (0,1/2)$, $z\in K$ we have:
 If $(z-\eps^s,z+\eps^s)$ contains at least $n+2$ points of $K$
 then $K\cap (z-\eps,z+\eps)\setminus (z-\eps^s,z+\eps^s)\neq\emptyset$.
 \end{theo}

\begin{proof}
Let $|\cdot|_n$ be a dominating norm on $C^{\infty}(K)$ and  $m=n+1$. We take $k>m$ and $c>0$ from the \DN-condition in submultiplicative form, i.e.,
\[
|f|^2_{n+1}\leq c|f|_{k}|f|_{n} \text{ for all } f\in C^{\infty}(K).
\]
Assume that, for all $s\ge 2$, there exist $z=x_0 \in K$ and $x_1,\ldots,x_{n+1}\in (x_0-\eps^s,x_0+\eps^s)$ so that
$K\cap (x_0-\eps,x_0+\eps)\setminus (x_0-\eps^s,x_0+\eps^s)=\emptyset$
(the dependence on $s$ is notationally suppressed).

We take $\varphi\in\DD(\R)$ with
$\supp(\varphi)\subset (x_0-\eps,x_0+\eps)$, $\varphi=1$ on the small interval $(x_0-\eps^s,x_0+\eps^s)$, and $|\varphi^{(j)}|\leq c_j\eps^{-j}$, where $c_j$ are
absolute constants. Let $P(x)=\prod_{j=0}^{n}(x-x_j)$ and $f=\varphi P$. Then
\[
  |f|_{n+1}\geq |P[x_0,\ldots,x_{n+1}]|=\left| \frac{P^{(n+1)}(\xi)}{(n+1)!}\right|=1.
  \]
From $f[x_0,\ldots,x_j]=f^{(j)}(\xi)/j!$ we further get
\begin{align*}
  |f|_k & \leq \sup\{|f^{(j)}(\xi)|:\ 0\leq j\leq k,\ \xi\in \R\}
  \\
\leq & \frac{c_k}{\eps^k} \sup_{x\in (x_0-\eps,x_0+\eps)}\sup_{0\leq j\leq k} \left|\sum_{|S|=n-j}\prod_{\ell \in S}(x-x_\ell)\right| \leq \frac{\tilde{c}}{\eps^k},
\end{align*}
the last estimate comes from $|x-x_\ell|\leq 1$ for each $x\in (x_0-\eps,x_0+\eps)$ so that $\tilde c$ only depends on $k$.

To estimate $|f|_n$ let us take distinct point $y_0,\ldots,y_n$ in $K$. Because of the symmetry of the divided differences we can assume that
$y_0 <y_1 <\cdots<y_n$. Leibniz' rule for the product $f=\varphi f$ says
\begin{align*}
  f[y_0,\ldots,y_n] &= \varphi[y_0]f[y_0,\ldots,y_n] + \varphi[y_0,y_1]f[y_1,\ldots,y_n]  \\ & +  \varphi[y_0,y_1,y_2]f[y_2,\ldots,y_n] + \cdots
\end{align*}
If $a$ is the first index with $y_a >x_0-\eps^s$ then the first $a-1$ terms of this sum vanish because $y_{a-1}$ is outside $(x_0-\eps,x_0+\eps)$ so that
$\varphi[y_0,\ldots,y_{a-1}]=0$. Estimating $|\varphi[y_0,\ldots,y_j]|= |\varphi^{(j)}(\xi_j)/j!| \le c_j \eps^{-j}$ we get
\begin{align*}
  |f[y_0,\ldots,y_n]|\le n \frac{c_n}{\eps^n} \sup\{|f[z_0,\ldots,z_j]|: z_0,\ldots,z_j \in K\cap (x_0-\eps^s,\infty),\ j\leq n\}.
\end{align*}
In the same way we estimate $|f[z_0,\ldots,z_n]|$ by $c_n\eps^{-n}$ times divided differences with nodes in $K\cap (x_0-\eps^s,x_0+\eps^s)$.
Since $f=P$ in $K\cap (x_0-\eps^s,x_0+\eps^s)$ this yields
\begin{align*}
|f|_n&\leq \frac{c_n}{\eps^{2n}} \sup\{|f[y_0,\ldots,y_\ell]|:\ 0\leq \ell \leq n,\, y_j\in  K\cap (z-\eps^s,z+\eps^s)\}\\
&= \frac{c_n}{\eps^{2n}} \sup\{|P[y_0,\ldots,y_\ell]|:\ 0\leq \ell \leq n,\, y_j\in  K\cap (z-\eps^s,z+\eps^s)\}\\
&\leq \frac{c_n}{\eps^{2n}} \sup\{|P^{(\ell)}(\xi)|:\ 0\leq \ell\leq n, \xi\in (z-\eps^s,z+\eps^s)\} \leq \frac{\tilde c_n}{\eps^{2n}}  \eps^{s}
\end{align*}
where $\tilde{c}$ is another constant which only depends on $n$. Taking the \DN-inequality together with the estimates obtained for $|f|_{n+1}$, $|f|_n$ and $|f|_k$,
we get for some constant $c$ which is independent of $s$ that
\[
1\leq c \eps^{s-2n-k}.
\]
For $s\to\infty$ this is impossible.
\end{proof}

\begin{example}\label{almostaccumulation}
The set $K=\{0\}\cup\{\frac1k+je^{-k}:\ 0\leq j \leq k,\ k\in\N\}$ does not have the smooth extension property.
\end{example}

\subsection{A necessary  Markov type inequality.}

We fix $\varphi\in \DD(\R)$ such that $\supp(\varphi)\subseteq [-1,1]$ and $\varphi=1$ in $[-1/2,1/2]$ and write,
$\varphi_{\eps,y}(x)=\varphi(\frac{x-y}{\eps})$ for $y\in \R$ and $\eps>0$.

\begin{proposition}
\label{necessarycondition}
If $\vertiii{\cdot}_n$ is a dominating norm on $C^{\infty}(K)$ then the following holds:

\noindent
$
\alle{m\in\N} \gibt{r\geq 1} \alle{k \in\N} \gibt{c_k>0} \text{ such that }
$
for all polynomials $P$ of degree $\le k$, $\eps>0$, accumulation points $y$ of $K$, and
$f\in C^{\infty}(\R)$ with $\supp(f)\subseteq (y-\eps,y+\eps)$ and
$f = \varphi_{\eps,y}P$ on $K$ we have
\[
|P^{(m)}(y)|^2\leq \frac{c_k}{\eps^r} \sup\limits_{|t-y|<\eps}|P(t)|\sup\limits_{|t-y|<\eps,\ j\leq n}|f^{(j)}(t)|\]

\end{proposition}
\begin{proof}
The  \DN condition implies  that,
for all $ m\geq n$, there are $r\geq m$ and $C>0$ such that for all $h\in C^{\infty}(K)$
$$
\vertiii{h}_m^2\leq C\vertiii{h}_{r}\vertiii{h}_n.
$$
Given the situation from the proposition we set $h =f |_K =\varphi_{\eps,y}P|_K$.
Since $y$ is an accumulation point of $K$ we have $g^{(m)}(y)=P^{(m)}(y)$ for all $g\in C^\infty(\R)$ satisfying  $g|_K=h$. Hence
$$
|P^{(m)}(y)|^2\leq \vertiii{h}_m^2\leq C\|\varphi_{\eps,y}P\|_{r}\|f\|_{n}.
$$
It remains to combine the Leibniz rule for $\varphi_{\eps,y}P$ with the estimate $|\varphi_{\eps,y}^{(i)}|\le c /\eps^i$ and the classical Markov inequality
\[
\sup\limits_{x\in [y-\eps,y+\eps]}|P^{(\ell)}(x)|\leq \frac{c}{\eps^\ell}\sup\limits_{t\in [y-\eps,y+\eps]}|P(t)|. \qedhere
\]
\end{proof}

The density $I_{K}^{2n+1}\subseteq \overline{I_K^{\infty}}^{C^n(\R)}$ (see the remarks after the proof of theorem
\ref{Omega}) allows us to write the $n$-th norm of $h\in C^{\infty}(K)$ as
$$
\vertiii{h}_n=\inf\{\|f\|_n:\ f\in C^{2n+1}(\R),\, f=h \text{ on } K \}.
$$
In the situation of the previous proposition we can thus replace $f\in C^\infty(\R)$ by $f\in C^{2n+1}(\R)$.

We now consider $K=\{0\}\cup \{a_\ell:\ \ell\in\N\}$ with $a_\ell \to 0$
to get examples where $C^{\infty}(K)$ does not satisfy  the condition in proposition \ref{necessarycondition}.

\begin{proposition}
\label{necessarysequencesdn}
Let $K=\{0\}\cup \{a_\ell: \ell\in\N\}$ with a null-sequence such that $(|a_\ell|)_{\ell\in\N}$ decreases and $|a_\ell| <|a_{\ell-1}|/2$.
If $\vertiii{\cdot}_n$ is a dominating norm for $C^{\infty}(K)$ then, for all $s\in\N$, there is $r\ge 1$ such that, for all $k\geq s$, there exists $C_k>0$ such that for all $d\in\N$
$$
\prod_{j=1}^{k-s}|a_{d+j}|^{4(2n+2)}\leq \frac{C_k}{|a_d|^r}|a_{d}|^{2k(2n+2)} |a_{d+k}|^{3n+4}.
$$

\noindent (For $k=s$ the empty set in the product of the left hand side is  1.)
\end{proposition}

\begin{proof}
Specifying in proposition \ref{necessarycondition} $m=2s(2n+2)$ and this concrete $K$ we  get:
If $\vertiii{\cdot}_n$ is a dominating norm on $C^{\infty}(K)$ then the following holds:

\noindent
$
\alle{s\in\N} \gibt{r\geq 1} \alle{k\in\N,\ k\geq s} \gibt{c_k>0} \text{ such that }
$
for all polynomials $P$ of degree $\le k$, $\eps>0$, and
$f\in C^{2n+1}(\R)$ with $\supp(f)\subseteq (-\eps,\eps)$ and
$f = \varphi_{\eps,0}P$ on $K$ we have

\begin{equation}
\label{necessarydn} \tag{$\ast$}
|P^{(2s(2n+2))}(0)|^2\leq \frac{c_k}{\eps^r} \sup_{|t-y|<\eps}|P(t)| \sup_{t\in\R,\ j\leq n}|f^{(j)}(t)|.
\end{equation}

For fixed $d\in\N$ we abbreviate $x_j=a_{d+j}$ with the quantifiers from the proposition. We consider $\eps=|x_0|$ and the polynomial
$$
P(x)=\prod_{j=1}^{k}(x^2-x_j^2)^{2n+2}
$$
of degree $\tilde k=2k(2n+2)$. Since $\pm x_k$ are zeroes of order $2n+2$ of $P$ the function
$$
f(x)=\left\{\begin{array}{ll} P(x)&\abs{x}\leq |x_k|\\ 0& \mbox{otherwise}\end{array}\right.,
$$
is in $C^{2n+1}(\R)$.
Moreover $f=\varphi_{\eps,0}P$ on $K$ because $\varphi_{\eps,0}(x)=1$ for $|x|\le |x_k|$ (as $|x_k|\le |x_0|/2=\eps/2$) and, for all other $x\in K\cap \supp(\varphi_{\eps,0})$,
we have $x\in\{x_1,\ldots,x_k\}$ so that $f(x)=P(x)=0$.

In order to apply (\ref{necessarydn}) to $f$ we will show the following inequalities:
\begin{itemize}
\item[($\alpha$)] For $m=2s(2n+2)$ we have
$\displaystyle |P^{(m)}(0)|^2\geq \prod_{j=1}^{k-s}|x_j|^{4(2n+2)},$

\item[($\beta$)] $\displaystyle \sup_{|t|\leq\eps}|P(t)|\leq |2x_0|^{2k(2n+2)}$,

\item[($\gamma$)] there are $c_k>0$ (depending only on $k$) such that
 $$\sup_{t\in\R}\sup_{0\leq\ell\leq n}|f^{(\ell)}(t)|\leq c_k|x_k|^{3n+4}.$$
\end{itemize}
This will imply the proposition.

\noindent $(\alpha)$ $P(x)$ is of the form
$\displaystyle
P(x)=\prod_{\ell=-N,\ell\neq 0}^{N}(x-y_\ell)
$
where $y_{-\ell}=-y_\ell$, $y_\ell\in \{x_1,\ldots,x_k\}$, $N=k(2n+2)$, and each $\pm x_j$ appears $2n+2$ times.
We have
$$
P^{(m)}(0)=\sum_{|S|=2N-m}\prod_{\ell\in S}(-y_\ell)
$$
(more precisely, we sum over all subsets $S$ of $\{-N,\ldots,N\} \setminus\{0\}$ with $2N-m=2(2n+2)(k-s)$ elements). 
We claim that all terms of the sum with non-symmetric $S$ (i.e. $-S\neq S$) cancel.
Indeed if $S$ is non symmetric we replace $\ell$ by $-\ell$, where $|\ell|$ is minimal such that $\ell\in S$, $-\ell\notin S$  to obtain
$\tilde{S}$ with $\prod_{\ell\in \tilde{S}}y_\ell=-\prod_{\ell\in S}y_\ell$.
For symmetric $S$ all terms have the same sign $(-1)^{\frac{|S|}{2}}$ and we can therefore estimate $|P^{(m)}(0)|$ from 
below by the absolute value of any term of the sum (since $k\ge s$ the sum is not empty, in the extreme case $k=s$ it contains just one term for $S=\emptyset$).
Choosing $S$ so that $\{y_\ell:\ \ell>0,\ \ell\in S \} =    \{x_1,\ldots,x_{k-s}\}$ we obtain
$$
|P^{(m)}(0)|\geq \prod_{j=1}^{k-s}|x_j|^{2(2n+2)}.
$$

\noindent $(\beta)$ follows from
$$
|x^2-x_j^2|=|(x+x_j)(x-x_j)|\leq (2x_0)^2\ \mbox{ for }|x|\leq |x_0|.
$$

\noindent $(\gamma)$ For $|t|\leq |x_k|$ and $j\leq n$ we have, with the same notation as in $(\alpha)$,
$$
f^{(j)}(t)=P^{(j)}(t)=\sum_{|S|=2N-j}\prod_{\ell\in S}(t-y_\ell).
$$
Since $x_k$ and $-x_k$ together appear $2(2n+2)$ times in $P$ among all $y_\ell$, at least $3n+4$ appear
in each product $\prod_{l\in S}(t-y_l)$ which is thus in absolute value smaller than $|2x_k|^{3n+4}$
(the number of terms $\binom{2N}{2N-j}$ only depends on $n$ and $j$).
\end{proof}

\begin{corollary}
Let $(a_\ell)_{\ell\in\N}$ be a null-sequence such that $(|a_\ell|)_{\ell\in\N}$ decreases and $|a_\ell| <|a_{\ell-1}|/2$. If $K=\{0\}\cup \{a_\ell: \ell\in\N\}$ has
the smooth extension property then there is $p\ge 1$ such that 
${a_\ell^p}/{a_{\ell+1}}$ is bounded.
\end{corollary}

\begin{proof}
Taking $s=k=1$ in Proposition \ref{necessarysequencesdn} we get $|{a_\ell^p}/{a_{\ell+1}}|\leq C_1$ for any $p \ge \frac{r-2(2n+2)}{3n+4}$.
\end{proof}

From this corollary we get immediately the following example.

\begin{example}
\label{examplefactorial}
The set $K=\{0\}\cup \{e^{-\ell!}:\ \ell\in\N\}$ does not have the smooth extension property.
\end{example}

We finish with an example of a convex sequence $(a_{\ell})$ showing that boundedness of 
${a_{\ell}^{p}}/{a_{\ell+1}}$ for some fixed $p>1$ is not enough for the smooth extension property.

\begin{example}
\label{exampleexponential}
$K=\{0\}\cup \{e^{-p^\ell}:\ \ell\in\N\}$ with $p>1$ does not have the smooth extension property.
\end{example}

\begin{proof}
Assume that $\vertiii{\cdot}_n$ is a dominating norm. We fix  $s\in\N$ such that
$$
p^{1-s}4(2n+2)<(3n+4).
$$
Proposition \ref{necessarysequencesdn} gives some $r\in\N$ such that, for each $k\in\N$, $k\geq s$, the sequence
$$
q_d=a_{d}^{r-2k(2n+2)}a_{d+k}^{-(3n+4)}\prod_{j=1}^{k-s}a_{d+j}^{4(2n+2)}
$$
is bounded with respect to $d$. We calculate
$$
\prod_{j=1}^{k-s}a_{d+j}=\exp\left(\sum_{j=1}^{k-s}p^{d+j}\right)=\exp\left(-p^{d+1}\frac{p^{k-s}-1}{p-1}\right),
$$
hence
\begin{align*}
 q_d & =\exp\left(2k(2n+2)-r)p^d+(3n+4)p^{d+k}-4(2n+2)p^{d+1}\frac{p^{k-s}-1}{p-1}\right) \\
     & =\exp\left(p^d\left((2k(2n+2)-r)+4(2n+2)\frac{p}{p-1}\right.\right. \\
     & \left.\left. \hspace{3cm} +p^k((3n+4)-p^{1-s}4(2n+2)\right)\right).
\end{align*}
  For $k$ big enough such that
 $2k(2n+2)-r$ is positive, the sequence is unbounded with respect to $d$.
\end{proof}

It is interesting to compare this example (for an integer $p\ge 3$) with a result of Goncharov \cite{Go} who proved that the somehow similar set
\[
\tilde K=\{0\}\cup \bigcup_{\ell\in\N} \left[e^{-p^\ell},e^{-p^\ell}-e^{-p^{\ell+1}}\right]
\]
does satisfy the Whitney (and hence also the smooth) extension property. This can be also seen as an application of Theorem \ref{sufficient}.

\subsection*{Acknowledgements.} The research of all authors is partially supported by GVA  AICO/2016/054 . 
The research of the second author was partially supported by  MINECO,  Project MTM2013-43540-P. 

\bibliographystyle{amsalpha}
\bibliography{ideals-lit}

\end{document}